\documentclass[reqno]{amsart}

\usepackage{enumitem}
\usepackage{hyperref}
\usepackage{bbm}
\usepackage{amsfonts,amsmath,amsxtra,amsthm,amssymb,latexsym}
\usepackage{ifpdf}
\usepackage{tikz}
\usepackage{color}

\newcommand{\ack}{\section*{Acknowledgments}}

\newcommand*{\mailto}[1]{\href{mailto:#1}{\nolinkurl{#1}}}
\newcommand{\arxiv}[1]{\href{http://arxiv.org/abs/#1}{arXiv:#1}}

\newcommand{\msc}[1]{\href{http://www.ams.org/msc/msc2010.html?t=&s=#1}{#1}}

\newtheorem{theorem}{Theorem}[section]
\newtheorem{corollary}[theorem]{Corollary}
\newtheorem{lemma}[theorem]{Lemma}

\newtheorem{remark}[theorem]{Remark}

\theoremstyle{definition}
\newtheorem{definition}[theorem]{Definition}

\newcommand{\be}{\begin{equation}}
\newcommand{\ee}{\end{equation}}
\newcommand{\ba}{\begin{array}}
\newcommand{\ea}{\end{array}}

\newcommand{\ol}{\overline}

\numberwithin{equation}{section}

\DeclareMathOperator{\Int}{int}

\newcommand\R{{\mathbb{R}}}

\newcommand\T{{\mathbb{T}}}

\newcommand\Z{{\mathbb{Z}}}

\newcommand\cT{{\mathcal{T}}}

\newcommand\cG{{\mathcal{G}}}
\newcommand\cE{{\mathcal{E}}}
\newcommand\cF{{\mathcal{F}}}
\newcommand\cP{{\mathcal{P}}}
\newcommand\cS{{\mathcal{S}}}
\newcommand\cV{{\mathcal{V}}}

\newcommand\bH{{\mathbf{H}}}

\def\wt#1{{{\widetilde #1} }}

\newcommand\mes{{\rm{mes}}}
\newcommand\comb{{\rm{comb}}}

\begin{document}

\title[Tessellating Quantum Graphs]{Strong Isoperimetric Inequality \\
For Tessellating Quantum Graphs}
	
\author[N. Nicolussi]{Noema Nicolussi}
\address{Faculty of Mathematics\\ University of Vienna\\
Oskar-Morgenstern-Platz 1\\ 1090 Vienna\\ Austria}
\email{\mailto{noema.nicolussi@univie.ac.at}}

\thanks{{\it Research supported by the Austrian Science Fund (FWF) 
under Grants P28807 and W1245.}}


\begin{abstract}
We investigate isoperimetric constants of infinite tessellating metric graphs. We introduce a curvature-like quantity, which plays the role of a metric graph analogue of discrete curvature notions for combinatorial tessellating graphs. Based on the definition in \cite{kn17}, we then prove a lower estimate and a criterium for positivity of the isoperimetric constant.
\end{abstract}

\maketitle
	
\section{Introduction} 
	{\em Isoperimetric constants}, which relate surface area and volume of sets, are among the most fundamental tools in spectral geometry of manifolds and graphs. They first appeared in this context in \cite{che}, where Cheeger obtained a lower bound on the spectral gap of Laplace--Beltrami operators. For discrete Laplacians on graphs, versions of Cheeger's inequality are known in various settings, e.g. \cite{al, am, bkw15, dod, dk86, fuj, km, moh88, moh92}. They find application in many fields (such as the study of expander graphs and random walks on graphs, we refer to \cite{lub} and \cite{woe} for more information) and consequently, there is a very large interest in graph isoperimetric constants. 
	
	In the case of {\em tessellating graphs} (i.e.\ edge graphs of tessellations of $\R^2$), they have been investigated using certain notions of {\em discrete curvature} (see for example \cite{h01, kp11, oh, woe, zuk}). On the other hand, the idea of plane graph curvature already appears earlier in several unrelated works \cite{g87, ish, s76} and was also employed to describe other geometric properties, for instance discrete analogues of the Gauss--Bonnet formula and the Bonnet--Myers theorem, e.g. \cite{bp01, dVm, HJL, k11, kp11, s76}.   
	
	Another framework for isoperimetric constants are {\em metric graphs} $\cG$, i.e.\ combinatorial graphs $\cG_d = (\cV, \cE)$ with vertex set $\cV$ and edge set $\cE$, where each edge $e \in \cE$ is identified with an interval $I_e = (0, |e|)$ of length $0 < |e|< \infty$ . Topologically, $\cG$ may be considered as a ``network" of intervals glued together at the vertices.  The analogue of the Laplace--Beltrami operator for metric graphs is the {\em Kirchhoff--Neumann Laplacian} $\bH$ (also known as a {\em quantum graph}). It acts as an edgewise (negative) second derivative $f_e \mapsto -\frac{d^2}{dx_e^2} f_e$, $e \in \cE$, and is defined on edgewise $H^2$-functions satisfying continuity and Kirchhoff conditions at the vertices  (we refer to \cite{bcfk06, bk13, ekkst08, post} for more information and references; see also \cite{ekmn, kn17} for the case that $\cG$ is infinite). A well-known result for {\em finite metric graphs} (i.e.\ $\cV$ and $\cE$ are finite sets) is a spectral gap estimate for $\bH$ in terms of an isoperimetric constant due to Nicaise \cite{nic} (see also \cite{kkmm16, kur}).  
	
	In this work, we are interested in {\em infinite metric graphs} (infinitely many vertices and edges). A notion of an isoperimetric constant $\alpha(\cG)$  in this context  was introduced recently in \cite{kn17} (see \eqref{eq:a_G} below for a precise definition) together with the following estimate
		\begin{align}\label{eq:cheeger}
		\frac{1}{4}\alpha(\cG)^2 \leq \lambda_0(\bH) \leq  \frac{\pi^2}{2\, \inf_{e\in \cE} |e| }\alpha(\cG),
		\end{align}
	which holds for every connected, simple, locally finite, infinite metric graph. 
	Here $\lambda_0(\bH) := \inf\, \sigma(\bH)$ is the bottom of the spectrum of $\bH$.

	However, let us stress that explicit computation of isoperimetric constants in general is a difficult problem (known to be NP-hard for combinatorial graphs \cite{moh89}). Hence the question arises, whether one can find bounds on $\alpha(\cG)$ in terms of less complicated quantities. On the other hand, the definition of $\alpha(\cG)$ is purely combinatorial and moreover $\alpha(\cG)$ is related to the isoperimetric constant $\alpha_{\comb} (\cG_d)$ of the combinatorial graph $\cG_d$ (see \cite{kn17} for further details). This strongly suggests to use discrete methods for further study. For combinatorial tessellating graphs, such tools are available in the form of discrete curvature and it is natural to ask whether similar techniques also apply to metric graphs.  Moreover, the class of plane graphs contains important examples such as trees and edge graphs of regular tessellations. 
	
	Motivated by this, the subject of our paper are isoperimetric constants of {\em infinite tessellating metric graphs} (see Definition \ref{def:tess}). Our main contribution is the definition of a {\em characteristic value} $c(\cdot)$ of the edges of a given metric graph (see \eqref{eq:charval}), which takes over the role of the classical discrete curvature  (up to sign convention; as opposed to e.g. \cite{ h01, k11,  kp11}, our results on $\alpha(\cG)$ are formulated in terms of positive curvature). In the simple case of equilateral metric graphs (i.e.\ $|e| = 1$ for all $e \in \cE$), $c$ coincides with the characteristic edge value introduced by Woess in \cite{woe98}. Moreover, for a finite tessellating metric graph (Corollary \ref{cor:GB}),		
	\begin{equation} 
	\sum_{e \in \cE} - c(e)|e| = 1,
\end{equation}
which can be interpreted as a metric graph analogue of the combinatorial Gauss--Bonnet formula known in the discrete case (see e.g. \cite{k11}).

	In terms of these characteristic values, we then formulate our two main results: Theorem \ref{prop:average} contains a criterium for positivity of $\alpha(\cG)$ based on the averaged value of $c(\cdot)$ on large subgraphs $\wt \cG \subset \cG$. In Theorem \ref{prop:est}, we obtain explicit lower bounds on $\alpha(\cG)$. A simplified version of this estimate is the following inequality:
	\begin{equation}
	\alpha(\cG) \geq \inf_{e \in \cE} c(e). 
	\end{equation}
	Theorem \ref{prop:est} can be interpreted as a metric graph analogue of the estimate in \cite[Theorem 1]{kp11} and a result by McKean in the manifold case \cite{mckean}. 
	
	Finally, we demonstrate the use of our theory by examples. First, we consider the case of equilateral  $(p,q)$-regular graphs. Here, $\alpha(\cG)$ is closely related to $\alpha_{\comb} (\cG_d)$  and hence can be computed explicitly. It turns out for large $p$ and $q$,  the estimate in Theorem \ref{prop:est} is quite close to the actual value. Second, we show how to construct an example where $\alpha(\cG)$ and $\alpha_{\comb} (\cG_d)$ behave differently.

	Let us finish the introduction by describing the structure of the paper.
	In Section 2, we recall a few basic notions and give a precise definition of infinite tessellating metric graphs. Moreover, we review the definition of $\alpha(\cG)$ and define the characteristic values. Section 3 contains our main results and proofs. In the final section, we consider examples. 	
	
	\section{Preliminaries}\label{sec:Prelim}
	
	\subsection{(Combinatorial) planar graphs}\label{ss:II.01}
	Let $\cG_d = (\cV, \cE)$ be an infinite, unoriented graph with countably infinite sets of vertices $\cV$ and edges $\cE$.  For a vertex $v \in \cV$, we set
	\be\label{eq:E_v}
	\cE_v := \{e \in \cE | \; e \text{  is incident to } v\}
	\ee
and denote by $\deg(v) := \# \cE_v$ the {\em combinatorial degree} of $v \in \cV$. Throughout the paper, $\# A$ denotes the number of elements of a given  set $A$. We will always assume that $\cG_d$ is connected, simple (no loops or multiple edges) and  locally finite ($\deg(v) < \infty$ for all $v \in \cV$). 

Moreover, we suppose that $\cG_d$ is {\em planar} and consider a fixed planar embedding. Hence $\cG_d$ can be regarded as a subset of the plane $\R^2$, which we assume to be closed. Denote by $\cT$ the set of {\em faces} of $\cG_d$, i.e.\ the closures of the connected components of $\R^2 \setminus \cG_d$. Note that unbounded faces of $\cG_d$ are included in $\cT$ as well. Motivated by the next definition, we will refer to the elements $T \in \cT$ as the {\em tiles} of $\cG_d$. 
  
\begin{definition}  \label{def:tess}
A planar graph $ \cG_d$ is {\em tessellating} if the following additional conditions hold:
\begin{enumerate}[label=(\roman*), ref=(\roman*), leftmargin=*, widest=iiii]
		\item $\cT$ is locally finite, i.e.\ each compact subset $K$ in $\R^2$ intersects only finitely many tiles.
		\item Each bounded tile $T \in \cT$ is a closed topological disc and its boundary $\partial T$ consists of a finite cycle of at least three edges.
		\item Each unbounded tile $T \in \cT$ is a closed topological half-plane and its boundary $\partial T$ consists of a (countably) infinite chain of edges.
		\item Each edge $e \in \cE$ is contained in the boundary of precisely two different tiles. 
		\item Each vertex $v \in \cV$ has degree $\geq 3$.
	\end{enumerate}
\end{definition}
Here, a subset $A \subseteq \R^2$ is called a closed topological disc (half-plane) if it is the image of the closed unit ball in $\R^2$ (the closed upper half-plane) under a homeomorphism $\phi \colon \R^2 \to \R^2$. For a tile $T \in \cT$, we define 
\begin{align}\label{eq:E_T}
\cE_T & := \{e \in \cE | \;  e \subseteq \partial T\}, & d_{\cT}(T) & := \# \cE_T,
\end{align}
where the latter is called the {\em degree of a tile} $T\in\cT$. Notice that according to Definition \ref{def:tess}(ii), $d_{\cT}(T) \geq 3$ for all $T \in \cT$.  

The above assumptions (i)--(v) imply that $\cT$ is a {\em locally finite tessellation} of $\R^2$, i.e.\ a locally finite, countable family of closed subsets $T \subset \R^2$ such that the interiors are pairwise disjoint and $\bigcup_{T \in \cT} T = \R^2$. In addition, $\cG_d = (\cV, \cE)$ coincides with the edge graph of the tessellation in the following sense: by calling a connected component of the intersection of at least two tiles $T \in \cT$ a $\cT$-vertex, if it has only one point and a $\cT$-edge otherwise, we recover the vertex and edge sets $\cV$ and $\cE$. 
	
For a finite subgraph $\wt \cG \subset \cG_d$, let $ \cF(\wt \cG)$ be the set of {\em bounded faces} of $\wt \cG$, i.e.\ the closures of all bounded, connected components of $\R^2 \backslash \wt \cG$. By local finiteness, each bounded face of $\wt \cG$ is a finite union of bounded tiles $T \in \cT$. Moreover, define $ \cP(\wt \cG)$ as the set of tiles $T \in \cT$ with $\partial T \subseteq  \wt   \cG$. Note that always
	\[  
	\cP(\wt \cG) \subseteq \cF(\wt \cG).
	\]
	
\subsection{Metric graphs} \label{ss:II.02}
After assigning each edge $e \in \cE$ a finite length $|e|\in (0,\infty)$, we obtain a {\em metric graph} $\cG:=(\cV,\cE,|\cdot| )= (\cG_d,|\cdot|)$. For a subgraph $\wt \cG = (\wt \cV, \wt \cE) \subset \cG$, define its {\em boundary vertices} by
	\begin{equation}\label{eq:bdG}
		\partial_\cG \wt{\cG} 	:= \big\{v\in \widetilde{\cV} |\  \deg_{\wt\cG}(v)<\deg(v)\big\},
	\end{equation}
	where $\deg_{\wt \cG}(v)$ denotes the degree of a vertex $v \in \wt \cV$ with respect to $\wt \cG$. For a given finite subgraph $\wt\cG\subset \cG$ we then set 
\be\label{eq:deg_wtG}
\deg(\partial_\cG \wt{\cG}) := \sum_{v\in\partial_\cG \wt{\cG}} \deg_{\wt\cG}(v).
\ee
Following \cite{kn17}, the {\em isoperimetric constant}  of a metric graph $\cG$ is then defined by
	\be\label{eq:a_G}
	\alpha(\cG) := \inf_{\wt{\cG}} \; \frac{ \deg (\partial_\cG \wt{\cG}) }{\mes ( \wt{\cG} ) } \in [0, \infty),
	\ee
	where the infimum is taken over all finite, connected subgraphs $\wt \cG \subset \cG$ and $\mes(\wt\cG)$ denotes the Lebesgue measure of $\wt\cG$, $ \mes(\wt\cG) := \sum_{e\in\wt\cE} |e|$. We will say that the metric graph $\cG$ satisfies the {\em strong isoperimetric inequality} if $\alpha(\cG)>0$. 
		
	Recall that for a combinatorial graph $\cG_d= (\cV, \cE)$ the {\em (combinatorial) isoperimetric constant} $\alpha_{\comb}(\cG_d)$ is defined by (see, e.g., \cite{dk86})
\be\label{eq:a_comb}
	\alpha_{\comb}(\cG_d) = \inf_{U \subset \cV} \frac{\# \{e \in \cE| \; e \text{ connects } U \text{ and } \cV \setminus U \}}{ \sum_{v \in U}\deg(v)},
\ee
where the infimum is taken over all finite subsets $U \subset \cV$. There is a close connection between $\alpha_{\comb}(\cG_d)$ and $\alpha(\cG)$ and we refer for further details to \cite{kn17}.

	We also need the following quantities. The {\em weight} $m(v)$ of a vertex $v \in \cV$ is given by
	\be\label{eq:def_m}
		m(v) = \sum_{e \in \cE_v} |e|.
	\ee
 Clearly, $m(v)$ equals the Lebesgue measure (the ``length") of the star $\cE_v$. 
 	The {\em perimeter} $p(T)$ of a tile $T \in \cT$ is defined as
	\begin{align}\label{eq:def_p}
		p(T) := \begin{cases} \sum_{e \in \cE_T} |e|, &   T \in \cT\ \text{is  bounded} \\  \infty, & T \in \cT \ \text{is unbounded} \end{cases}.
	\end{align}

For every $e\in\cE$, we define its {\em characteristic value} $c(e)$  by
	\begin{align}\label{eq:charval}
		c(e) := \frac{1} { |e| }- \sum_{v: v \in e} \frac{1} {m(v)} - \sum_{T: e \subseteq \partial T} \frac{1}{p(T)}.
	\end{align}
Here we employ the convention that whenever $\infty$ appears in a denominator, the corresponding fraction $1/p$ has to be interpreted as zero if $p$ is infinite. Let us mention that for equilateral metric graphs $\cG$ (i.e.\ $|e| \equiv 1$ for all $e \in \cE$), the characteristic value $c(e)$ coincides with the characteristic edge value introduced in \cite{woe98} in the context of combinatorial graphs. 
	
Finally, we need the following class of subgraphs of $\cG$.  
A subgraph $\wt \cG =(\wt \cV, \wt \cE) \subset \cG$ is called {\em star-like}, if it can be written as a finite, connected union of stars. More precisely,
	\[ \wt \cE = \bigcup_{v \in \wt U}\cE_v  \]
	for some finite, connected vertex set $\wt U \subseteq \wt{\cV}$.
	
	Also, for a finite subgraph $\wt \cG \subset \cG$, we define its {\em interior graph} $\wt \cG_{\Int} = (\wt \cV_{\Int},\wt  \cE_{\Int} ) $ as the set of interior vertices $v \in \wt \cV_{\Int} := \wt \cV \setminus \partial \wt \cG$ together with all edges between such vertices.  We say that $\wt \cG$ is {\em complete}, if $ \cF(\wt \cG_{\Int}) = \cP(\wt \cG_{\Int})$, or equivalently if every bounded face of $\wt \cG_{\Int}$ consists of exactly one tile $T \in \cT$. Let us denote the class of star-like complete subgraphs by $\cS(\cG)$.
	
\section{Strong isoperimetric inequality for tessellating quantum graphs}\label{sec:main}
		
Now we are in position to formulate our main results. Our first theorem relates the positivity of the isoperimetric constant with the positivity of the characteristic values of a metric graph. 

\begin{theorem} \label{prop:average}
	Let $\cG = (\cV, \cE, |\cdot|)$ be a tessellating metric graph having infinite volume, $\mes(\cG) = \sum_{e\in\cE}|e|=\infty$. If 
	\be\label{eq:ellast}
	\ell^\ast(\cG):= \sup_{e \in \cE} |e| < \infty
	\ee
	 and
		\be\label{eq:aver_c}	
		\liminf_{\substack{\mes(\wt \cG) \to \infty }} \frac{1}{ \mes(\wt \cG)}\sum_{e \in \wt \cE} c(e) |e|
		= \liminf_{\substack{\mes(\wt \cG) \to \infty }} \frac{\sum_{e \in \wt \cE} c(e) |e|}{\sum_{e \in \wt \cE} |e|}
		    > 0,
		\ee
	then $\alpha(\cG) >0$. Here $\liminf$ is taken over all star-like complete subgraphs $\wt \cG\in\cS(\cG)$.
\end{theorem}

\begin{remark}
A few remarks are in order.
\begin{enumerate}[label=(\roman*), ref=(\roman*), leftmargin=*, widest=iiii]
\item[(i)] Let us mention that \eqref{eq:ellast} is necessary for the strong isoperimetric inequality to hold for an arbitrary metric graph since (see, e.g., \cite[Remark 3.3]{kn17})
\be\label{eq:estSol}
\alpha(\cG)  \le \frac{2}{\ell^\ast(\cG)}.
\ee
\item[(ii)] If $\mes(\cG) = \sum_{e\in\cE}|e| < \infty$, then the lower bound
\begin{equation} \label{eq:estvol}
	\alpha(\cG) \geq \frac{2}{\mes(\cG)}>0
\end{equation}
holds. In fact, if $\cG$ is tessellating, then $\deg (\partial \wt \cG) \ge 2$ for every finite subgraph $\wt \cG \subset \cG$ and \eqref{eq:estvol} follows immediately from \eqref{eq:a_G}. 
\item[(iii)] Theorem \ref{prop:average} can be seen as the analogue of \cite[Theorem 1]{woe98}. 
\item[(iv)] As we will see below, the proof of Theorem \ref{prop:average} implies the explicit estimate
\begin{equation} \label{eq:est01}
	\alpha (\cG) \geq \min \Big\{\frac{2}{\ell^\ast(\cG)}, \; \ \inf_{\wt \cG \in \cS}  \frac{1}{ \mes(\wt \cG)}\sum_{e \in \wt \cE} c(e) |e|  \Big\},
\end{equation}
however, the conditions in Theorem \ref{prop:average} are weaker than positivity of the right-hand side in \eqref{eq:est01}.
\end{enumerate}
\end{remark}

The next result shows that pointwise estimates for the characteristic values also yield lower estimates for the isoperimetric constant. To this end, introduce 
the following quantities
\begin{align}
M(\cG) & := \sup_{v\in\cV} \frac{m(v)}{\inf_{e\in\cE_v}|e|}, & P(\cG) & := \sup_{T \in \cT} \frac{p(T)}{\inf_{e\in \cE_T}|e|},
\end{align}
and set
\be\label{eq:K_G}
	 K  (\cG) := 1-  \frac{1}{M(\cG)}  -  \frac{2}{P(\cG)} - \frac{1}{(M(\cG) -2) P(\cG)}.
	\ee

	\begin{theorem} \label{prop:est}
	Let $\cG = (\cV, \cE, |\cdot|)$ be a tessellating metric graph. Suppose
	\be\label{eq:ce_ast}
	c_\ast(\cG) := \inf_{e \in \cE} c(e) > 0.
        \ee
        Then
	\be\label{eq:iso_02}
		\alpha(\cG) \geq    \frac{c_\ast (\cG)}{K (\cG)}  > 0. %{1 - \frac{1}{m_s^\ast( \cG)} - \frac{2} {p_s^\ast ( \cG)}  } >0.
	\ee
\end{theorem}
Theorem \ref{prop:est}  can be considered as the metric graph analogue of the corresponding estimate for combinatorial graphs in \cite[Theorem 1]{kp11}.
\begin{remark} \label{rem:K}
The following obvious estimates
\begin{align} \label{est:Mdeg}
M(\cG) & \ge \sup_{v\in\cV} \deg(v) \ge 3, & P(\cG) & \ge \sup_{T\in\cT} d_\cT(T) \ge 3,
\end{align}
imply that $K(\cG) \le 1$. 
Moreover, one can show that $K(\cG) > 0$ if $c_\ast(\cG) > 0$. Indeed, noting that
\begin{align*}
m(v) & \le  \deg(v)\ell^\ast(\cG), & p(T) & \le  d_\cT(T)\ell^\ast(\cG),
\end{align*}
we easily get the following rough estimate
\begin{align}\label{eq:c_ast_est1}
	c_\ast (\cG) &\leq 
		 \frac{1}{\ell^\ast (\cG)}   \Big( 1 - \frac{2}{ \deg^\ast (\cG) } -  \frac{2}{   d_\cT^\ast (\cG) }    \Big),
\end{align}
where $\deg^\ast (\cG)  :=  \sup_{v\in\cV} \deg(v)$ and $  d_\cT^\ast (\cG)   :=  \sup_{T\in\cT} d_\cT(T)$.
On the other hand,
\begin{align*}
	K(\cG) & =  \frac{1}{2}\Big(1 -  \frac{2}{M(\cG)}  -  \frac{2}{P(\cG)}\Big)+ \frac{1}{2} -  \frac{1}{P(\cG)} - \frac{1}{(M(\cG) -2) P(\cG)} \\
	&\ge \frac{1}{2}\Big( 1 - \frac{2}{ \deg^\ast (\cG) } -  \frac{2}{   d_\cT^\ast (\cG)  }    \Big)  + \frac{1}{2} - \frac{1}{   d_\cT^\ast (\cG) } - \frac{1}{( \deg^\ast (\cG)- 2)  d_\cT^\ast (\cG) }.
\end{align*}
If $c_\ast(\cG) > 0$, then so is the right-hand side in \eqref{eq:c_ast_est1} which implies 
\begin{align*}
	K(\cG)  > & \frac{1}{2} - \frac{1}{  d_\cT^\ast (\cG) } - \frac{1}{( \deg^\ast (\cG) - 2)   d_\cT^\ast (\cG) }\\
	 > & \frac{1}{ \deg^\ast (\cG) } - \frac{1}{( \deg^\ast (\cG) - 2)  d_\cT^\ast (\cG) } \ge 0.
\end{align*}
\end{remark}

To prove Theorems \ref{prop:average} and \ref{prop:est}, we first show that we can restrict in \eqref{eq:a_G} to star-like complete subgraphs.
\begin{lemma} \label{lem:reduction}
	Let $\cG = (\cV, \cE, |\cdot|)$ be a tessellating metric graph. Then
	\begin{align} \label{eq:minalpha}
		\alpha(\cG) = \min \Big\{\frac{2}{\ell^\ast(\cG)}, \alpha_\cS(\cG)  \Big\},
	\end{align}
	where 
	\be\label{eq:aS_G}
		\alpha_\cS(\cG) := \inf_{\wt \cG \in \cS }  \frac{\deg ( \partial \wt \cG)}{ \mes(\wt \cG)}.
	\ee
	\end{lemma}
	\begin{proof}
	(i) First, we show that it suffices to consider subgraphs that are either star-like or consist of a single edge. Let $\wt \cG = (\wt \cV, \wt \cE)$ be a finite, connected subgraph of $\cG$ and $\wt \cV_{\Int} = \wt \cV \setminus \partial \wt \cG$. We split $ \wt \cV_{\Int} = \bigcup_{i=1}^n \cV_i$ into a finite, disjoint union of connected vertex sets $\cV_i$. Let $ \cG_i =(\cV_i, \cE_i) \subset \cG$ be the subgraph with edge set
	\[
		 \cE_i = \bigcup_{v \in \cV_i} \cE_v.
	\]
	 By construction, each $ \cG_i$ is star-like and each edge $e \in  \wt \cE$ belongs to at most one $ \cG_i$. Let 
	 $\wt\cE_r = \wt \cE \setminus \bigcup_{i=1}^n \cE_i$ be the remaining edges. Then 
	 \[
	 \mes(\wt \cG) = \sum_{i=1}^n \mes(\cG_i)  + \sum_{e \in  \wt\cE_r} |e|.
	 \]
	Moreover, both vertices of an edge $e \in \wt \cE_r$ are in $\partial \wt \cG$ and $\partial \cG_i = \partial \wt \cG \cap \cV_i$. Hence
	\begin{align*}
	\deg(\partial \wt{\cG}) = \sum_{v \in \partial \wt \cG}  \sum_{i=1}^n \deg_{\cG_i} (v) +  2 \# \cE_r = \sum_{i=1}^n \deg(\partial \cG_i) + 2 \# \cE_r.
	\end{align*}
	This finally implies
	\begin{align*}
		\frac{ \deg(\partial \wt{\cG} )}{\mes(\wt \cG)} = \frac { \sum_{i=1}^n \deg(\partial \cG_i) + 2 \# \cE_r } { \sum_{i=1}^n \mes(\cG_i)  + \sum_{e \in  \cE_r} |e|} \geq \min_{\substack{i=1,...,n, \\ e \in \cE_r}} \left \{  \frac{ \deg(\partial \cG_i  )}{ \mes(\cG_i)}, \frac{2}{ | e |} \right \}.
	\end{align*}

	(ii) To complete the proof, it suffices to construct for every star-like subgraph $\wt \cG $ a star-like, complete subgraph $\widehat \cG \in\cS(\cG)$ with $\widehat \cG  \supseteq   \wt \cG  $ and $ \deg ( \partial  \wt \cG)  \geq  \deg ( \partial  \widehat \cG) $. 
	Let $\wt\cG_{\Int}=(\wt \cV_{\Int}, \wt \cE_{\Int})$ be the interior graph of $\wt \cG$. Denote by $  \cF_0  $ the set of bounded, open components of  $\R^2 \setminus \wt \cG_{\Int}$ and by $\cF = \{  F = \ol{f} | \; f \in \cF_0\}$ the bounded faces of $\wt \cG_{\Int}$. The idea is to add ``edges contained in bounded faces". Define the  subgraph $\widehat \cG = (\widehat \cV, \widehat \cE)$ by its edge set
	\[
		\widehat \cE =  \wt \cE \cup \bigcup_{v \in f\colon f \in \cF_0} \cE_v   .
	\]
	If an edge $e \in \cE$ is incident to a vertex $v \in f$ with $ f \in \cF_0$, then its other vertex lies in $F = \ol{f}$. Hence $\deg_{\widehat \cG} (v) = \deg_{\wt \cG} (v) $ for every vertex $v$ with $v \notin K := \bigcup_{F \in \cF }F$. On the other hand, every vertex $v \in K$ belongs to $\widehat \cG$ and satisfies $\deg_{\widehat \cG} (v) = \deg_{\cG} (v)$. Indeed, if $v \in  F = \ol{f}$, then either $v \in f$ or $v \in \partial f \subseteq \wt \cG_{\Int} $.  
	This implies $\partial \widehat \cG \subseteq \partial \wt \cG$ and $ \deg ( \partial  \wt \cG)  \geq  \deg ( \partial  \widehat \cG) $. 
	Moreover, $\widehat \cE = \cup_{v \in \widehat U} \cE_v$, where 
	\[
		\widehat U = \wt \cV_{\Int} \cup \bigcup_{f \in \cF_0} \{ v \in \cV | \; v \in f \}.
	\]
	 Also, $\widehat U$ is finite by local finiteness and connected since $\wt \cG$ is star-like and $\partial f \subseteq \wt \cG_{\Int}$ for $f \in \cF_0$. It remains to show that $\widehat \cG$ is complete. Let $\widehat F$ be a bounded face of the interior graph $\widehat \cG_{\Int}$. Suppose $T \in \cT$ with $T \subseteq \widehat F$. Then $T \subseteq \widehat F \subseteq F$ for some bounded face $F$ of $\wt \cG_{\Int}$. In particular, $e \subseteq K$ for every edge $e \subseteq \partial T$. But every vertex $ v \in K$ belongs to $\widehat \cG_{\Int}$, and hence $\partial T \subseteq \widehat \cG_{\Int}$ and $\widehat F = T$
\end{proof}

\begin{remark} \label{rem:min}
Combining \eqref{eq:minalpha} with \eqref{eq:estSol}, one concludes that the inequality 
\begin{equation} \label{eq:estmin}
\frac{2}{\ell^\ast(\cG)} \le  \alpha_\cS(\cG) = \inf_{\wt \cG \in \cS }  \frac{\deg ( \partial \wt \cG)}{ \mes(\wt \cG)}
\end{equation}
implies that 
\be
\alpha(\cG) = \frac{2}{\ell^\ast(\cG)}.
\ee
In Example \ref{ex:min}, we provide an explicit construction of a graph satisfying \eqref{eq:estmin}.
\end{remark}

The next lemma contains the connection between $c(e)$ and $\alpha(\cG)$.
\begin{lemma} \label{lem:degsum}
		The following inequality 
		\be\label{eq:isoStar}
			\sum_{e \in \wt \cE} c(e) |e| \leq \deg (\partial_\cG \wt \cG)
		\ee
		holds for any star-like, complete subgraph $\wt \cG\in \cS(\cG)$.		
	\end{lemma}
	
\begin{proof}
	Let $\wt \cG_{\Int} = (\wt \cV_{\Int}, \wt \cE_{\Int})$ be the interior graph and  $\cE^b := \wt \cE \setminus \wt \cE_{\Int}$ the remaining edges. Then
 		\begin{align*}
		\sum_{e \in \wt \cE} c(e) |e| &= \sum_{ e \in \wt \cE} 1 - \sum_{v \in \wt \cV}   \frac{\mes(\cE_v \cap \wt \cE)}{ m(v)} - \sum_{T \in \cT} \frac{\mes( \cE_T \cap \wt \cE_{\Int})}  {p(T)} - \sum_{T \in \cT} \frac{\mes( \cE_T \cap \cE^b)}{p(T)}    \\
		&= \# \cE^b  +  \# \wt \cE_{\Int}   - \# \wt \cV_{\Int}  -   \# \cP(\wt \cG_{\Int})	\\
		& \; \; -   \sum_{v \in \partial \wt \cG} \frac{\mes(\cE_v \cap \wt \cE )}{ m(v)}- \sum_{T \in \cT, \cE_T \not \subseteq \wt \cE_{\Int}}  \frac{\mes( \cE_T \cap \wt \cE_{\Int})} {p(T)}  - \sum_{T \in \cT} \frac{\mes( \cE_T \cap \cE^b)}{p(T)}, 
	\end{align*}
	where $\cP(\wt \cG_{\Int})$ is the set of tiles $T \in \cT$ with $\cE_T \subseteq \wt \cE_{\Int}$. By Euler's formula (see, e.g., \cite{har})
	\[
		  \# \wt \cE_{\Int} - \# \wt \cV_{\Int}  - \#  \cF(\wt \cG_{\Int})  = - 1,
	\]
Because $\wt \cG$ is complete,  $\cF(\wt \cG_{\Int}) = \cP(\wt \cG_{\Int})$ and $\sum |e| c(e)$ is equal to
	\begin{align*}
		 & \# \cE^b -1 -  \sum_{v \in \partial \wt \cG} \frac{\mes(\cE_v \cap \wt \cE )}{ m(v)} - \sum_{T \in \cT, \cE_T \not \subseteq \wt \cE_{\Int}}  \frac{\mes( \cE_T \cap \wt \cE_{\Int})} {p(T)}  - \sum_{T \in \cT} \frac{\mes( \cE_T \cap \cE^b)}{p(T)}.
	\end{align*}
 Since $\wt \cG$ is star-like, there are no edges $e \in \wt \cE$ with both vertices in $\partial \wt \cG$. Therefore, $
  \# \cE^b = \deg( \partial \wt \cG)$ and the proof is complete.
\end{proof}

\begin{remark}

For future reference, observe that 
	\begin{align*}
  	\sum_{v \in \partial \wt \cG} & \frac{\mes(\cE_v \cap \wt \cE )}{ m(v)} + \sum_{T \in \cT, \cE_T \not \subseteq \wt \cE_{\Int}}  \frac{\mes( \cE_T \cap \wt \cE_{\Int})} {p(T)}  + \sum_{T \in \cT} \frac{\mes( \cE_T \cap \cE^b)}{p(T)} \\
	& \geq \sum_{v \in \partial \wt \cG} \frac{\deg_{\wt \cG} (v) }{ M(\cG)  } +   \sum_{T \in \cT, \cE_T \not \subseteq \wt \cE_{\Int}} \frac{ \# (\cE_T \cap \wt \cE_{\Int}) }{ P(\cG) }  +   \sum_{T \in \cT} \frac{ \# ( \cE_T \cap \cE^b )}{ P(\cG) }.
  \end{align*}
   This implies the following estimate
  \begin{equation} \label{eq:tech}
  \begin{split}
  	\sum_{e \in \wt \cE} c(e) |e| \leq&  \deg (\partial \wt \cG)\Big( 1 - \frac{1}{M(\cG)} - \frac{2}{P(\cG) }  \Big)  \\
	&- \frac{1}{P(\cG) }  \sum_{ e \in \wt \cE_{\Int}} \#  \big\{ T | \; e \in \cE_T \text{ and } \cE_T \not \subseteq \wt \cE_{\Int} \big\} 
	\end{split}
  \end{equation} 
  for every  star-like, complete subgraph $\wt \cG \in \cS(\cG)$.
\end{remark}

\begin{corollary}\label{cor:GB}
Let  $\cG = (\cV, \cE, |\cdot|)$ be  a finite tessellating metric graph, that is a finite graph satisfying all the assumptions of Section \ref{ss:II.01} except (iii) of Definition \ref{def:tess}. 
Then
	\begin{equation} \label{eq:GB1}
		\sum_{e \in \cE} - c(e)|e| = 1.
	\end{equation} 
\end{corollary}

\begin{proof}	
	By Euler's formula
	\begin{align*}
		\sum_{e \in \cE} c(e)|e|  &= \# \cE - \sum_{v \in \cV} \frac{\mes(\cE \cap \cE_v) }{m(v)} - \sum_{T \in \cT}  \frac{\mes(\cE \cap \cE_T) }{ p(T) } \\
						&= \# \cE - \# \cV - \# \cF(\cG) = -1. \qedhere
	\end{align*}
\end{proof}	
\begin{remark} \label{rem:GB}
Formula \eqref{eq:GB1}  can be seen as the analogue of the combinatorial Gauss--Bonnet formula known for combinatorial graphs (see \cite[Proposition 1]{k11}). Let us also mention that the difference in the right-hand side arises from our convention  $p(T) = \infty$ for unbounded $T \in \cT$. 
\end{remark}

Theorem \ref{prop:average} now follows from Lemma \ref{lem:reduction} and \ref{lem:degsum} together with the inequality $\deg( \partial \wt \cG) \geq 1$ for $\wt \cG \in \cS(\cG)$. Moreover, we can already deduce (see \eqref{eq:est01} and \eqref{eq:c_ast_est1}) the basic estimate
\begin{equation} \label{eq:estbasic}
	\alpha(\cG) \geq c_\ast(\cG).
\end{equation}
By improving this bound further we finally obtain Theorem \ref{prop:est}.

\begin{proof}[Proof of Theorem \ref{prop:est}]
We start by providing a basic inequality. By \ref{rem:K}, we have $K(\cG) >0$. Let $\deg^\ast(\cG) := \sup_{v \in \cV} \deg(v)$. Then using \eqref{est:Mdeg} and \eqref{eq:c_ast_est1}, a lengthy but straightforward calculation implies
\begin{equation} \label{eq:fundamentalest}
	\frac{c_\ast(\cG)}{K(\cG)}  \leq \frac{\deg^\ast (\cG)-2}{\deg^\ast (\cG)-1} \frac{1}{\ell^\ast(\cG)}.
\end{equation}
Hence by Lemma \ref{lem:reduction}, it suffices to show that 
\begin{equation} \label{eq:toprove}
	\frac{\deg(\partial \wt \cG)}{\mes(\wt \cG)} \geq \frac{c_\ast(\cG)}{K(\cG)}
\end{equation}
for every $\wt \cG= (\wt \cV, \wt \cE) \in \cS(\cG)$.
 
 We will obtain \eqref{eq:toprove} by induction over $\#\wt \cV_{\Int} $. If $\# \wt \cV_{\Int}  = 1$, that is, $\wt \cV_{\Int}  = \{v\}$ for some $v\in\cV$, then $\wt \cG$ ``consists of a single star". More precisely, $\wt \cE = \cE_v$ and \eqref{eq:fundamentalest} implies
\[
	\frac{\deg(\partial \wt \cG)}{\mes(\wt \cG)} \geq \frac{\deg(v)}{\deg(v) \ell^\ast(\cG)} \geq  \frac{c_\ast(\cG)}{K(\cG)} .
\]
Now suppose $\# \wt \cV_{\Int}  = n \geq 2$ and \eqref{eq:toprove} holds for all $\widehat \cG \in \cS(\cG)$ with $ \#\widehat \cV_{\Int}  < n$. We distinguish two cases: 

(i) First, assume 
\[
\#\{u \in \partial \wt \cG| \; u \text{ is connected to } v \} \leq \deg^\ast (\cG)-2
\]
 for all $v \in \wt \cV_{\Int} $. In view of \eqref{eq:tech}, define 	
 \[
 \cE_i := \{e \in \wt \cE_{\Int}  |\; \# \{ T | \; e \in \cE_T \text{ and } \cE_T \not \subseteq  \wt \cE_{\Int}\} = i \}
 \]
for $i\in\{1,2\}$. Then
\begin{align*}
\sum_{ e \in \wt \cE_{\Int}} {\#\{ T | \; e \in \cE_T \text{ and } \cE_T \not \subseteq \wt \cE_{\Int} \} } = \# \cE_1 + 2 \# \cE_2 = \sum_{v \in \wt \cV_{\Int} } \delta(v) ,
\end{align*}
where $\delta(v):= \#( \cE_v \cap \cE_1) /2 +  \#( \cE_v \cap \cE_2)$ for all $v \in \cV$.

Now assume that $v \in \wt \cV_{\Int} $ and that $v$ is connected to at least one vertex in $ \partial \wt \cG$. Since $\wt \cG$ is star-like and $\# \wt \cV_{\Int}  \geq 2 $, $v$ is connected to another vertex in $\wt \cV_{\Int} $ and hence there exists an interior edge $e \in \wt \cE_{\Int} $ incident to $v$. Going through the edges incident to $v$ in counter-clockwise direction starting from $e$, denote by $e_+$ the ``last" edge incident to $v$ with $e_+ \in \wt \cE_{\Int} $. Define $e_-$ analogously by going clockwise. Then $e_\pm \in \cE_1 \cup \cE_2$. Moreover, if $e_+ = e_-$, then $e = e_+ = e_- \in \cE_2$. Thus $\delta(v) \geq 1$ for every such $v \in \wt \cV_{\Int} $. Since $\wt \cG$ is star-like,
\begin{align*}
	\sum_{v \in \wt \cV_{\Int} } \delta(v) &\geq \frac{1}{\deg^\ast (\cG)-2} \sum_{v \in \wt \cV_{\Int} }  \# \{u \in \partial \wt \cG|\; u \text{ is connected to } v\} \\
	&\geq \frac{1}{M (\cG)-2} \deg(\partial \wt \cG),
\end{align*}
	and \eqref{eq:toprove} follows from \eqref{eq:tech}. 
	
	(ii) Assume that $\#\{u \in \partial \wt \cG| \; u \text{ is connected to } v \} \geq \deg^\ast (\cG) -1 $ for some vertex $v \in \wt \cV_{\Int} $. Since $\# \wt \cV_{\Int} \geq 2$, this implies $\deg(v) = \deg^\ast (\cG)$ and that $v$ is connected to exactly one $w \in \wt \cV_{\Int} $. We ``cut out" the $\deg^\ast (\cG)-1$ edges between $v$ and $\partial \wt \cG$ and define $\widehat \cG = (\widehat \cV, \widehat \cE)$ by its edge set
	\[
		\widehat \cE = \wt \cE \setminus \{e \in \cE| \; e \text{ connects } v \text{ and } \partial \wt \cG \}.
	\]
	Then $\widehat \cG$ is again star-like and complete. Its interior graph $\widehat \cG_{\Int}  = (\widehat \cV_{\Int} , \widehat \cE_{\Int} )$ is given by $\widehat \cV_{\Int}  = \wt \cV_{\Int}  \setminus \{v\}$ and $\widehat \cE_{\Int}  = \wt \cE_{\Int}  \setminus \{e_{v,w}\}$,
	where $e_{v,w}$ is the edge between $v$ and $w$. In particular, $\widehat \cG$ satisfies \eqref{eq:toprove}. 
	
	Now assume $\eqref{eq:toprove}$ fails for $\wt \cG$. Then 
	\[
	\deg(\partial \wt \cG) (\mes(\wt \cG) - \mes(\widehat \cG)) \leq (\deg^\ast (\cG)-2)  \mes(\wt \cG)
	\]
	by \eqref{eq:fundamentalest}. Consequently,
	\begin{align*}
		\frac{\deg(\partial \widehat \cG)}{\mes(\widehat \cG)} = \frac{ \# \widehat \cE - \# \widehat \cE_{\Int}  }{\mes(\widehat \cG)} &= \frac{ \# \wt \cE - \# \wt \cE_{\Int}  - \deg^\ast (\cG) + 2 }{\mes(\widehat \cG)}  \\
		&= \frac{\deg(\partial \wt \cG) -\deg^\ast (\cG) +2}{\mes(\widehat \cG)} \leq \frac{\deg(\partial \wt \cG)}{\mes(\wt \cG)} < \frac{c_\ast(\cG)}{K(\cG) },
	\end{align*}
	which is a contradiction. 
\end{proof}

%%%%%%%%%%%%%%%%%%%%%%%%%%%%%%%%%%%%%%%%%%%%%%%%%%%%
%%%%%%%%%%%%%%%%%%%%%%%%%%%%%%%%%%%%%%%%%%%%%%%%%%%%
%%%%%%%%%%%%%%%%%%%%% EXAMPLE SECTION %%%%%%%%%%%%%%%%%%%%
%%%%%%%%%%%%%%%%%%%%%%%%%%%%%%%%%%%%%%%%%%%%%%%%%%%%
%%%%%%%%%%%%%%%%%%%%%%%%%%%%%%%%%%%%%%%%%%%%%%%%%%%%

\section{Examples}\label{sec:examples}

In this section, we illustrate the use of our results in three examples.

\subsection{(p,q)-regular graphs} \label{ex:pq}
Let $p\in \Z_{\ge 3}$ and  $q\in \Z_{\ge 3}\cup \{\infty\}$. A tessellating (combinatorial) graph $\cG_d=(\cV, \cE)$ is called {\em $(p,q)$-regular}, if $\deg(v) =p$ for all $v \in \cV$ and $d_{\cT}(T)  = q$ for all $T \in \cT$. Let $\cG_{p,q}$ denote both the corresponding combinatorial graph and the associated equilateral  metric graph, that is,  we put $|e| \equiv 1$ for all $e\in\cE_{p,q} = \cE(\cG_{p,q})$. Notice that $\cG_{p,\infty}$ is an infinite $p$-regular tree $\T_p$ (also known as a {\em Cayley tree} or a {\em Bethe lattice}). 

Next, by \eqref{eq:charval},  we get 
\begin{equation}\label{eq:c_pq}
	c(e) = 1- \frac{2}{p} - \frac{2}{q} =: c_{p,q},
\end{equation}
for all $e \in \cE_{p,q}$, and the {\em vertex curvature} of the combinatorial graph $\cG_{p,q}$ (see for example \cite{h01, dVm, woe98}) is given by
\begin{align}\label{eq:kappa_pq}
	 \kappa(v) = 1 - \frac{\deg(v)}{2} + \sum_{T: v \in T} \frac{1}{  d_{\cT}(T) } = 1 - \frac{p}{2} + \frac{p}{q} = -\frac{p}{2}\,c_{p,q},
\end{align}
for all $ v \in \cV$. 

Since strictly positive vertex curvature implies that $\cG_d$ has only finitely many vertices (see \cite[Theorem 1.7]{dVm}), the characteristic value should satisfy $c_{p,q} \geq 0$.
 Clearly, $c_{p,q}=0$ exactly when $(p,q) \in \{ (4,4), (3,6), (6,3)\}$ and in these cases $\cG_{p,q}$ is isomorphic to the square, hexagonal or triangle lattice in $\R^2$. If $c_{p,q} >0$, then $\cG_{p,q}$ is isomorphic to the edge graph of a tessellation of the Poincar\'e disc ${\rm H}^2$ with regular $q$-gons of interior angle $2 \pi/p$ (see \cite[Remark 4.2.]{hjl} and \cite{iv}). In the latter case, Theorem \ref{prop:est} implies $\alpha(\cG_{p,q})> 0$ and the estimate 
\begin{align} \label{eq:estpq}
	\alpha(\cG_{p,q}) \geq  \frac{  q(p-2)c_{p,q} }{q(p-1)c_{p,q} + 1 } 
	=  \frac{p-2}{p-1}\times \begin{cases} \frac{1}{1+ (q(p-1)\,c_{p,q})^{-1}}, &q < \infty \\[2mm] 1, &q =  \infty\end{cases}. 
\end{align}
 Notice that in the case $q = \infty$, equality holds true in \eqref{eq:estpq} (see, e.g., \cite[Example 8.3]{kn17}).

 It is well-known that (see \cite{hs03}, \cite{hjl}),
\begin{equation} \label{eq:alphacomb}
	\alpha_{\comb} (\cG_{p,q}) = \frac{p-2}{p} \sqrt{1 - \frac{4}{(p-2)(q-2)}}.
\end{equation}
By (a slight modification of)  \cite[Lemma 4.1]{kn17}, 
\begin{align} \label{eq:combmetric}
	\alpha(\cG) = \frac{2 \alpha_{\comb} (\cG_d) }{\alpha_{\comb}( \cG_d )+1}
\end{align}
for every equilateral metric graph $\cG=(\cV, \cE, | \cdot |)$ with underlying combinatorial graph $\cG_d = (\cV, \cE)$. Hence
\be\label{eq:a_Gpg}
	\alpha(\cG_{p,q}) = \frac{p-2}{p-1 + \frac{p}{2} \left( \sqrt{ \frac{(p-2)(q-2)}{pq-2(p+q) }} - 1 \right)} 
	=  \frac{p-2}{p-1} \times 
	\begin{cases} \; \frac{  1 }{1 +  \delta (q(p-1)\, c_{p,q}) ^{-1}   },  &q < \infty \\[2 mm]1, &q =  \infty \end{cases},
	\ee
where 
\[
\delta := \frac{pq-2(p+q)}{2} \left( \sqrt{ 1 + \frac{4}{pq-2(p+q)}} - 1 \right) \leq 1.
\]
Comparing \eqref{eq:a_Gpg} with \eqref{eq:estpq}, we conclude that the error in the estimate \eqref{eq:estpq} is uniformly of order $\frac{1}{(pq)^2}$.

Finally, let us mention that using \eqref{eq:combmetric}, we can turn \eqref{eq:estpq} into a lower estimate for $\alpha_{\comb}(\cG_{p,q})$ as well. After a short calculation, we recover Theorem 1 from \cite{kp11},
\begin{equation}
	\alpha_{\comb}(\cG_{p,q}) \geq \frac{p-2}{p} \left( 1 - \frac{2}{(p-2)(q-2)-2} \right).
\end{equation}

\subsection{Another example} \label{ex:c}
Denote by $\Z^2_+ $ the square lattice of the upper half-plane, i.e.\ the combinatorial graph with vertex set $\Z \times  \Z_{\geq 0}$ and two vertices connected if and only if they are connected in the square lattice $\Z^2= \Z \times \Z$. Fix $k \in \Z_{\ge 3}$ and let $\cG_k$ be the graph obtained from $\Z_+^2$ by attaching to each vertex $v \in \Z \times \{ 0 \}$ an infinite $k$-regular tree (see Figure~\ref{fig:G1}). 

To assign edge lengths, we first define a partition of the edge set $\cE_{k}$.  We denote by $\cE_{k,{\rm tree}}$ the set of edges $e \in \cE_k$ belonging to one of the attached trees. Also,  let 
\[
 \cV_n = \{(z, n)| \; z \in \Z \} = \Z \times \{n\},\qquad n\in\Z,
 \]
be the vertices on the ``$n$-th horizontal line". 
For $n \in \Z_{\ge 0}$, we define $ \cE_{k,n}^+$ as the set of ``vertical" edges between the $n$-th horizontal line $\cV_n$ and the $(n+1)$-th horizontal line $\cV_{n+1}$, and $ \cE_{k,n}^-$ as the set of ``horizontal" edges connecting vertices in the $n$-th horizontal line $\cV_n$ (see Figure \ref{fig:G1}).  
Finally, we equip $ \cG_k$ with edge lengths in the following way:  
\be\label{eq:length2}
|e| = \begin{cases} 1, & e\in \cE_{k,{\rm tree}} \\[1mm] \frac{1}{(2n+2)^2}, & e\in \cE^-_{k,n} \\[1mm] \frac{1}{(2n+3)^2}, & e\in \cE^+_{k,n} \end{cases}.
\ee

\begin{figure} [ht]
	\begin{center}
		\begin{tikzpicture}    [scale=1.3]
	%%% VERTICES
		\foreach \x in {-3,...,3}
		\foreach \y in {0,...,2}
		{ \filldraw  ( \x ,\y) circle (1 pt);  }

	%%% PLACING EDGES
	\foreach \y in {0,...,2}{
		\draw (-3.5,  \y) -- (3.5 ,  \y) ;
		\draw [ dotted] (-4,  \y) -- (-3.5 ,  \y) ;
		\draw [ dotted] (3.5,  \y) -- (4 ,  \y) ;
	}
	
	\draw (-4,  0) -- (4 ,  0) ;

	\foreach \x in {-3,...,3}{
		\draw ( \x, 0) -- (\x, 2.5) ;
		\draw [ dotted] ( \x, 2) -- (\x, 3) ;
	}

	%%% LABEL EDGE SETS
	%NORMAL
	\foreach \y in {0,1, 2}{
		\node [right] at (-4.6, \y ) { $  \cE_{k, \y}^-$};
		\node [right] at (-3.8, \y + 0.5 ) { $  \cE_{k, \y}^+$};
	}
	
	%TREE EDGES
	\node [right] at (-4.6, - 1 ) {$  \cE_{k,{\rm tree}}$};
	
	%%%% DRAW attached trees	
	\foreach \x in {-3,...,3}{
	\foreach \y in {-1,...,1}{
	\foreach \z in {-1,1}{
	\foreach \w in {-1,1}{	
		\draw ( \x, 0) -- (\x -1/3*\y, -0.5) ;
		{ \filldraw (\x -1/3*\y, -0.5) circle (1 pt);  }
		\draw (\x -1/3*\y, -0.5) -- (\x -1/3*\y + 1/12*\z, -1) ;
		{ \filldraw (\x -1/3*\y + 1/12*\z, -1) circle (1 pt);  }
		\draw [densely dotted]  (\x -1/3*\y + 1/12*\z, -1) -- (\x -1/3*\y + 1/12*\z + 1/20*\w, -1.5) ;
	}}}}

		\end{tikzpicture}
	\end{center}
	\caption{$\cG_k$ for $k = 3$.} \label{fig:G1}
\end{figure}

Now let us compute the characteristic values. First of all, for all $e\in \cE_{k,{\rm tree}}$ we have the estimate
\[
 \inf_{e\in\cE_{k,{\rm tree}}} c(e) = 1 - \frac{2}{k} = \frac{k-2}{k}.
\]
Next, taking into account that $k\ge 3$, we get 
\[
	c(e) =  4 - \frac{2}{ k + 2\frac{1}{4} + \frac{1}{9} } - \frac{1}{\frac{1}{4} + 2\frac{1}{9}+\frac{1}{16} } = \frac{164}{77} - \frac{2}{ k + \frac{11}{18}} > 1
\]
for all $e \in \cE_{k,0}^-$, and 
\[ 
	c(e) = 9 - \frac{1}{ k + 2\frac{1}{4} + \frac{1}{9} } - \frac{1}{\frac{1}{9}+\frac{1}{25} + 2\frac{1}{16}} - \frac{2}{\frac{1}{4} + 2\frac{1}{9}+\frac{1}{16} }= \frac{8955}{ 5467}  -    \frac{1}{ k + \frac{11}{18}} > 1
\]
for $ e \in \cE_{k,0}^+$. 
%\begin{align*}
%	&c(e) = 4  - \frac{2}{ k + \frac{1}{2} + \frac{1}{9} } - \frac{1}{ \frac{1}{4}  + \frac{2}{9}  + \frac{1}{16}  } = \frac{164}{77} - \frac{2}{ k + \frac{1}{2} + \frac{1}{9} } \geq 1,  && e \in \wt \cE_2 \\
%	&c(e) = 9 - \frac{1}{ k + \frac{1}{2} + \frac{1}{9} } - \frac{1}{ \frac{1}{9} + \frac{1}{8} + \frac{1}{25}  } - \frac{2}{ \frac{1}{4} + \frac{2}{9} + \frac{1}{16}}  = \frac{8555}{ 5467}  -    \frac{1}{ k + \frac{1}{2} + \frac{1}{9} } \geq 1,  && e \in \wt \cE_3
%\end{align*}
Moreover, after lengthy but straightforward calculations one can see that
\[
c(e) > 1
\]
for all $e \in \cE_{k,n}^\pm$ with $n \geq 1$.
Thus we obtain
\be
	c_\ast(\cG_k) = \inf_{e\in \cE_{k,{\rm tree}}} c(e) = \frac{k-2}{k} >0,
\ee
and hence, by Theorem \ref{prop:est}, $\alpha(\cG_k)  >0$.  

Now let us compute $K(\cG_k)$. If $v \in \cV_n$ with $n \geq 1$, then
	\begin{align*}
		\sup_{v\in \cup_{n\ge 1}\cV_n} \frac{m(v)}{\inf_{e \in \cE_v}|e|} &= \sup_{n\ge 1}\,(2n+3)^2 \left( \frac{1}{(2n+1)^2} +  \frac{2}{(2n+2)^2} + \frac{1}{(2n +3)^2} \right) \\
								&=\left (1 + \frac{2}{3} \right)^2     + 2     \left (1 + \frac{1}{4} \right)^2   + 1  = \frac{497}{72} = 6.902\dot 7
	\end{align*}
	For $v \in \cV_0$, we obtain	\[
		 \frac{m(v)}{\inf_{e \in \cE_v}|e|}  = 9 \left( k + 2 \frac{1}{4} + \frac{1}{9} \right) = 9k + \frac{11}{2}
	\]
	Moreover, for the remaining vertices $v \in \cV$ belonging to one of the attached trees,
	\[
		\frac{m(v)}{\inf_{e \in \cE_v} |e|} = k.
	\]
	By assumption, $k \geq 3$ and hence $M(\cG_k) = 9k + \frac{11}{2}$. In addition, $P(\cG_k) = \infty$ since $\cT $ contains unbounded tiles. Thus we obtain
	\be\label{eq:K_Gk}
		K(\cG_k) = 1 - \frac{1}{M(\cG_k)} = \frac{18 k + 9}{18k + 11},
	\ee
	and Theorem \ref{prop:est} implies the lower estimate
	\[
		\alpha(\cG_k) \geq \frac{18 k +11}{18 k + 9} \frac{k-2}{k}.
	\]
	Our next goal is to derive an upper estimate. For $l \in \Z_{\ge 2}$, let $\wt \cG_l$ be the subgraph consisting of all edges in one of the attached trees between $\cV_0$ and $\cV_{-l}$.  Then it is straightforward to verify
	\begin{align*}
		& \mes(\wt \cG_l) = \sum_{j=0}^{l-1} k(k-1)^j = \frac{ k ( (k-1)^{l} - 1) }{k-2}  , &&\deg (\partial \wt \cG_l) = k + 3 + k(k-1)^{l-1},
	\end{align*}
	and as a consequence,
	\[
		\lim_{l \to \infty} \frac{ \deg (\partial \wt \cG_l)}{ \mes(\wt \cG_l) }= \frac{k-2}{k-1} = \alpha(\T_k),
	\]
	where $\T_k$ is the equilateral, $k$-regular tree (see Example \ref{ex:pq} or \cite[Example 8.3]{kn17}). This implies the two-sided estimate
	\[
	 	\frac{18 k +11}{18 k + 9} \frac{k-2}{k} \leq \alpha(\cG_k) \leq  \frac{k-2}{k-1}.
	\]
	In particular, $\alpha(\cG_k) \to 1$ for $k \to \infty$.
	
\begin{remark}
The two above examples demonstrate the use of Theorem \ref{prop:est} in two different situations. First of all, let us mention that by \cite[Corollary 4.4.]{kn17} the metric graph $\cG$ satisfies the strong isoperimetric inequality if $\ell^\ast(\cG)<\infty$ and the combinatorial isoperimetric constant $\alpha_{\comb}(\cG_d)$ is positive,
\[
 \alpha_{\comb}(\cG_d) >0. 
\]
In Example \ref{ex:pq}, the positivity of $\alpha_{\comb}(\cG_{p,q})$ is known (see \eqref{eq:alphacomb}) and hence it is a priori clear that $\alpha(\cG) > 0$. However,  Example \ref{ex:pq} shows that in certain situations Theorem \ref{prop:est} gives a good quantitative estimate.  

On the other hand, in Example \eqref{eq:alphacomb} we have $\alpha_{\comb}(\cG_k) = 0$ (since clearly $\alpha_{\comb}(\Z^2_+) = 0$), however, $\alpha(\cG) > 0$. In particular, Theorem \ref{prop:est} shows that the isoperimetric constants of the combinatorial and metric graph behave differently.
\end{remark}

\subsection{Non-equilateral $p$-regular trees} \label{ex:min}
We conclude with an example showing the use of Remark \ref{rem:min}.
For $p \in\Z_{\ge 5}$, let $\T_p$ be the equilateral, $p$-regular tree from Example \ref{ex:pq}.  Fix an edge $\hat e \in \cE(\T_{p}) $. In the following, we will consider $\T_p$ equipped with another choice of edge lengths.  Define the metric graph $\cT_p:=(\T_p,|\cdot|)$ by assigning 
\[
 |e| := \begin{cases} p, &e = \hat e \\[1mm] 1, &e \in \cE(\T_{p}) \setminus \{ \hat e \} \end{cases}.
\]

Let $\wt \cG \in \cS(\cT_p)$ be a star-like complete subgraph. If $\hat e \notin \wt \cE$, then $\mes (\wt \cG) = \# \wt\cE $. If $\hat e \in \wt \cE$, then $\wt \cV_{\Int} \neq \varnothing$ since $\wt \cG$ is star-like. Hence
\[
	\mes(\wt \cG) =\# \wt\cE+ p-1 \le 2 \# \wt\cE.
\]
 Thus we conclude from \eqref{eq:a_Gpg} and Lemma \ref{lem:reduction} that
\[
	\alpha_\cS( \cT_p)  = \inf_{\wt \cG \in \cS }  \frac{\deg ( \partial \wt \cG)}{ \mes(\wt \cG)} \geq \frac{1}{2}  \inf_{\wt \cG \in \cS }  \frac{\deg ( \partial \wt \cG)}{\#\wt\cE} = \frac{1}{2} \alpha( \T_p) = \frac{1}{2} \frac{p-2}{p-1} \ge \frac{2}{5}
\]
for all $p\ge 6$. 
On the other hand, $\ell^\ast(\cT_p) = p \geq 6$ by assumption. Hence Remark \ref{rem:min} implies
\[
	\alpha(  \cT_p) = \frac{2}{\ell^\ast (  \cT_p)} = \frac{2}{p}.
\]

%%%%%%%%%%%%%%%%%%%%%%%%%%%%%%%%%%%

\noindent
\ack I thank Aleksey Kostenko for helpful discussions throughout the preparation of this article and Delio Mugnolo for useful comments and hints with respect to the literature.

\end{document}